\documentclass[12pt]{amsart}

\usepackage{amsmath,amssymb}
\usepackage{tikz}
\usepackage{graphicx}
\usepackage{float}
\usetikzlibrary{positioning,calc}
\usepackage{comment}
\usepackage{siunitx}
\usepackage{array} 
\usepackage{url}
\usepackage{hyperref}
\usepackage{a4wide}

\newcommand{\md}{{\rm d}}
\newcommand{\dx}{\, \md x}

\newcommand{\R}{\mathbb{R}}
\newcommand{\N}{\mathbb{N}}

\newcommand{\om}{\Omega}

\newcommand{\cof}{{\rm cof}\,}

\newcommand{\im}{\mathrm{im}\,}

\newcommand{\D}{\mathbb{D}}

\def\XXint#1#2#3{{\setbox0=\hbox{$#1{#2#3}{\int}$}
     \vcenter{\hbox{$#2#3$}}\kern-.5\wd0}}

\newtheorem{definition}{Definition}[section]
\newtheorem{lemma}[definition]{Lemma}
\newtheorem{theorem}[definition]{Theorem}
\newtheorem{proposition}[definition]{Proposition}
\newtheorem{corollary}[definition]{Corollary}
\newtheorem{remark}[definition]{Remark}

\newcommand{\referee}{\color{black}}

 \def\FS#1#2#3#4 {\scalebox{0.7}{{\begin{tabular}{|c|c|} \hline $#2$ &$#1$ \\ \hline $#3$ & $#4$ \\ \hline \end{tabular}}}}
 \def\INS#1#2#3#4 {\scalebox{0.8}{{\begin{tabular}{|c|c|c|c|} \hline $#1$ &$#2 $ & $#3$ & $#4$   \\ \hline \end{tabular}}}}

\title{ Latent convexity as a tool implying  uniqueness results for polyconvex integrands}

\begin{document}

\author[J. Bevan]{Jonathan J. Bevan}
\address[J. Bevan]{School of Mathematics and Physics, University of Surrey, Guildford, GU2 7XH, United Kingdom. }
\email{j.bevan@surrey.ac.uk}

\author[M. Kru\v{z}\'{i}k]
{Martin Kru\v{z}\'{i}k}
\address[M. Kru\v{z}\'{i}k]{ Czech Academy of Sciences, Institute of Information Theory and Automation,
Pod vod\'arenskou v\v{e}\v z\'\i\ 4, 182 00, Prague 8, Czechia $\&$
Department of Physics, Faculty of Civil Engineering, Czech Technical University in Prague, Thákurova 7, 166 29 Prague 6, Czechia.
}
\email{kruzik@utia.cas.cz}

\author[J. Valdman] {Jan Valdman} 
\address[J. Valdman]{Czech Academy of Sciences, Institute of Information Theory and Automation, Pod vod\'arenskou v\v{e}\v z\'\i\ 4, 182 00, Prague 8, Czechia 
}
\email{jan.valdman@utia.cas.cz}


\title[Mean Hadamard inequalities and convex functionals]{Sharp mean Hadamard inequalities and polyconvex integrands that give rise to convex functionals}
\begin{abstract} 
We investigate several instances of the Hadamard inequality in the mean in two dimensions. As a consequence, we prove the uniqueness of minimizers of an integral functional with a polyconvex integrand, subject to mixed Dirichlet and Neumann boundary conditions. The theoretical findings are complemented by computational experiments that illustrate the behavior of the minimizers.\end{abstract}
\maketitle
\section{Introduction}

Let $\Omega$ be a bounded Lipschitz domain in $\R^2$ and let $f\in L^\infty(\Omega)$.  This paper is concerned with the convexity of the functional
\begin{align}\label{eye}I(\varphi)& = \int_{\Omega} |\nabla \varphi|^2+ f(x)\det \nabla \varphi \dx \end{align}
defined on $W_0^{1,2}(\Omega;\R^2)$, which given its quadratic dependence on $\varphi$ is equivalent to the condition 
\begin{align}\label{meanHad}
     \int_{\om} |\nabla \varphi|^2 + f(x)  \det \nabla \varphi \, \dx \geq 0 \quad  \forall \varphi \in W_0^{1,2}(\Omega;\R^2).
    \end{align}
Inequality \eqref{meanHad} is immediately implied by the pointwise Hadamard inequality for matrices $$|A|^2 \geq 2|\det A| \quad A  \in \R^{2 \times 2} $$ 
for $f$ with the property that 
\begin{align}\label{strongcondition} |f(x) - \langle f \rangle_{\om}| \leq 2 \quad \mathrm{a.e.} \ x \in \om,\end{align}
which follows by coupling \eqref{strongcondition} with the well-known fact that the function $A \mapsto \det A$ is a null Lagrangian, cf.~e.g.~\cite{dacorogna}.   Here, 
$$  \langle f \rangle_{\om}:=\frac{1}{\mathcal{L}^2(\om)}\int_{\om} f(x) \, \dx$$
denotes the mean value of $f$ over the domain $\om$, and $\mathcal{L}^2(\om)$ is the two-dimensional Lebesgue measure of $\om$.

In fact, \eqref{meanHad} can also be shown to hold even when condition \eqref{strongcondition} fails, which is how \eqref{meanHad} earns its name of a `Hadamard-in-the-mean' inequality, such as when $f=M \chi_{\om'}$ for any $M$ with $|M|\leq 4$ and $\om'$ any reasonable subdomain of $\om$.  See \cite[Proposition 3.4]{BKV23} for details, including the use of ideas on quasiconvexity at the boundary due to Mielke and Sprenger \cite{mielke-sprenger} that are needed to show that the condition $|M|=4$ is sharp.  

In this and in other examples, we stress that although the integrand
$$ W(x,A):= |A|^2 + f(x) \det A $$
is polyconvex \cite{Ba77}, it does not automatically follow that $I(\varphi) \geq I(0)=0$ for any $\varphi \in W_0^{1,2}(\om,\R^2)$.  One reason is that we may not assume that 
$$\int_{\om}W(x,\nabla \varphi) \, \dx \geq \int_{\om}W(x,0)\, \dx$$
holds in the case of an $x$-dependent quasiconvex integrand. Another reason is that through \cite[Proposition 6.2]{BKV23}, a clear link is made between the sequential lower semicontinuity of $I$ in $W_0^{1,2}(\om,\R^2)$ and the nonnegativity of $I(\varphi)$ for $\varphi \in W_0^{1,2}(\om,\R^2)$.  Nor are standard devices such as studying solutions of the Euler-Lagrange equations of any use.  The positivity or otherwise of $I(\varphi)$ therefore has to be decided by other means.

Although we aim for a general characterization of those $f$ for which \eqref{meanHad} holds, the approach we have taken in \cite{BKV23} and \cite{BKV25} has necessarily focused on establishing \eqref{meanHad} for certain key examples of $f$, chief amongst which is the Hadamard-in-the-mean inequality 
\begin{align}\label{ins:c}
    \int_{R_{-2}} |\nabla u|^2 - c \det \nabla u \, \dx + \int_{R_{-1}\cup R_{1}} |\nabla u|^2 \, \dx + 
    \int_{R_{2}} |\nabla u|^2 + c \det \nabla u \, \dx \geq  0
\end{align}
for all $u \in W_0^{1,2}(\om,\R^2)$, where $|c| \leq 4$ and
\begin{align}\label{Omega}
\Omega := R_{-2} \cup R_{-1} \cup R_1 \cup R_2
\end{align}
is formed of four rectangles arranged in a row, as shown in Figure~\ref{pic:pureinsulation}. 

\begin{figure}[H]
\centering
\begin{minipage}{0.65\textwidth}
\centering
\begin{tikzpicture}[scale=0.8]
\node (A) at (-4,-2) {}; 
\node[right=2 of A.center] (B) {};
\node[right=4 of A.center] (O) {};
\node[right=4 of B.center] (C) {};
\node[right=2 of C.center] (D) {};

\node[above=4 of A.center] (A') {};
\node[above=4 of B.center] (B') {};
\node[above=4 of O.center] (O') {};
\node[above=4 of C.center] (C') {};
\node[above=4 of D.center] (D') {};

\node[right=2 of B.center] (BC) {};
\node[right=2 of B'.center] (B'C') {};

\filldraw[thick, top color=gray!30,bottom color=gray!30] (A.center) rectangle node{$-c$} (B'.center);
\filldraw[thick, top color=gray!30,bottom color=gray!30] (C.center) rectangle node{$+c$} (D'.center);

\node (R1) at ($(B)!0.5!(B'C')$) {0};
\node (R-1) at ($(BC)!0.5!(C')$) {0};

\node [below right = 0.1 of A'.center] {$R_{-2}$};
\node [below right = 0.1 of B'.center] {$R_{-1}$};
\node [below right = 0.1 of O'.center] {$R_{1}$};
\node [below right = 0.1 of C'.center] {$R_{2}$};

\draw[thick] (B.center) -- (C.center);
\draw[thick] (B'.center) -- (C'.center);
\draw[thick] (BC.center) -- (B'C'.center);
\end{tikzpicture}
\end{minipage}
\hfill
\begin{minipage}{0.32\textwidth}
\small
\begin{align*}
R_{-2} &= (-1, -\tfrac{1}{2}) \times (-\tfrac{1}{2}, \tfrac{1}{2}), \\
R_{-1} &= (-\tfrac{1}{2}, 0) \times (-\tfrac{1}{2}, \tfrac{1}{2}), \\
R_{1}  &= (0, \tfrac{1}{2}) \times (-\tfrac{1}{2}, \tfrac{1}{2}), \\
R_{2}  &= (\tfrac{1}{2}, 1) \times (-\tfrac{1}{2}, \tfrac{1}{2}).
\end{align*}
\end{minipage}
\caption{Distribution of rectangles.}
\label{pic:pureinsulation}
\end{figure}
In terms of the functional in \eqref{eye}, the corresponding weight function $f$ is 
\begin{align*}
    f(x) = -c \chi_{R_{-2}} + c \chi_{R_{2}},
\end{align*}
and the central region $R_{-1}\cup R_{1}=(-\frac{1}{2},\frac{1}{2})\times (-\frac{1}{2},\frac{1}{2})$ `insulates' the regions $R_{-2}$ and $R_2$, where $f$ is non-zero, from one another.  Demonstrating \eqref{ins:c} is therefore dubbed the `insulation problem', and it was shown in \cite[Proposition 4.5]{BKV23} that \eqref{ins:c} holds for $c$ in the range $(2,2+\epsilon_0)$ for some $\epsilon_0$.  One of the main results of this paper, Theorem \ref{Thm:insulation}, is that \eqref{ins:c} holds for any $c$ such that $|c|\leq 4$, where the upper bound of $4$ is sharp, and that for $|c|\leq 4$, the unique minimizer of the functional in \eqref{ins:c} is $u=0$.  The technique we use takes advantage of the symmetries of $\om$ and of the functional itself: see Section \ref{insuproper} for the details.  In Theorem \ref{Th:delta} and Proposition \ref{asymptotic}, we consider the effect on $c$ of varying the width of the insulation layer.  The theme of both of these results is that the nonnegativity of the adjusted functional is maintained as long as $|c-2|$ is subordinated to the width of the insulation layer.  The results are not sharp, in contrast to Theorem \ref{Thm:insulation}. 

One can also view the functional in \eqref{eye} as a general form of an `excess functional' associated with an energy $E$ and a suitably-defined stationary point $u_0$, say, so that 
\begin{align}\label{splitting} E(u) = E(u_0)+I(\varphi),\end{align}
with $\varphi=u-u_0$. This is the situation discussed in  \cite{Bevan-14} and \cite{BD22} where, in both cases, the functional $I$ is of the form \eqref{eye}, $f=C\ln(|\cdot|)$, $C$ is constant, and the domain of integration is the unit ball in $\R^2$.  For large enough $C$, \cite[Proposition 3.5 (i)]{Bevan-14} shows that \eqref{meanHad} fails;  by contrast, it can be deduced from \cite[Theorem 1.2]{BD22} that, for sufficiently small $C$, \eqref{meanHad} holds.

 We can take $E=I$ in \eqref{splitting}.  Indeed, if 
 $u=u_0$ on $\Gamma_D\subset \partial \om$ with $u_0\in W^{1,2}(\om; \R^2)$ given then 
 \begin{align}\label{sum}
I(u_0+\varphi)= I(u_0) + \langle I'(u_0),\varphi\rangle +I(\varphi)\ , 
\end{align}
where $I$ is defined in \eqref{eye} and $\varphi\in W^{1,2}(\om;\R^2)$ such that $\varphi=0$ on $\Gamma_D$.  Note that if $u_0$ solves the Euler-Lagrange equations  of  $E$ in the weak sense, i.e., if $\langle E'(u_0),\varphi\rangle=0$  and 
$I(\varphi)\ge 0$ for every $\varphi$ defined above, then $u_0$ is a global  minimizer of $E$. 
This also means that \eqref{sum} implies that for every $\varphi\in W^{1,2}(\om;\R^2)$ such that $\varphi=0$ on $\Gamma_D$ it holds
\begin{align}
    I(u_0+\varphi)\ge I(u_0) + \langle I'(u_0),\varphi\rangle, 
\end{align}
i.e., $I$ is convex.
If we can show that $I(\varphi)=0$ only if $\varphi=0$, then $u_0$ is the unique minimizer of $I$, and $I$ is strictly convex.
This idea leads, in the case $n=2$, to a new technique for finding global minimizers of the Dirichlet energy $\D(u)=\int_{\om} |\nabla u|^2 \dx$ in classes where the Jacobian $\det \nabla u$ is \emph{a priori} prescribed pointwise a.e., enabling us to solve the sort of constrained minimization problem 
that typically arises in incompressible nonlinear elasticity theory (where $g \equiv 1$), but with one important difference.  This is that having first prescribed boundary data $u_0$ and a suitable pressure $f$ in
\begin{align}\label{itwo}I(u):=\int_{\Omega} |\nabla u|^2 + f(x) \,\det \nabla u \, \dx,\end{align}
the data $g$ in the Jacobian constraint emerges (rather than being prescribed \emph{a priori}) as $g:=\det \nabla U$, where $U$ minimizes $I(\cdot)$ in the unconstrained class $W^{1,2}_{u_{0}}(\om,\R^2)$, i.e., the Sobolev space with $u_0$ on $\partial\om$. 

Using this technique, we are able to prove, for example, that for suitable constants $\zeta$ and $\xi$ the map 
\begin{align}\label{examplediskdisk}
    u(x):=\left\{\begin{array}{l l} \zeta x & \ \ x \in B(0,\rho) \\
    \left(\xi+\frac{1-\xi}{|x|^2}\right)x & \ \ x \in B(0,1)\setminus B(0,\rho) \end{array}\right.
    \end{align}
is the unique global minimizer of the Dirichlet energy in 
 $   \left\{v \in W_{\textrm{id}}^{1,2}(B(0,1);\R^2): \ \det \nabla v = g \ \textrm{a.e.}\right\}, $
where $g(x):=\zeta^2$ if $x\in B(0,\rho)$ and $g(x):=\xi^2-(1-\xi)^2 |x|^{-4}$ otherwise.  Here, $B(0,\rho)$ stands, as usual, for the ball in $\R^2$ centered at $0$ and of radius $\rho$; see Figure \eqref{problem:disk-disk}. 

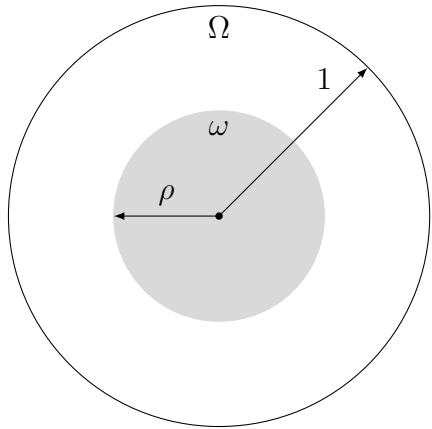
\begin{figure}[H]
\centering
\begin{minipage}{1\textwidth}
\centering
\begin{tikzpicture}[scale=1.4,>=latex]

\draw (0,0) circle (2.0cm);             
\draw[gray!30!,fill] (0,0) circle(1.0cm); 

\fill (0,0) circle (1pt); 

\node[below] at (0,1.0) {$\omega$}; 
\node[below] at (0,2.0) {$\Omega$}; 

\draw[->] (0,0) -- (-1.0,0); 
\node[above] at (-0.5,0) {$\rho$}; 

\draw[->] (0,0) -- (1.41,1.41); 
\node[above=4pt] at (1.0,1.0) {$1$}; 

\end{tikzpicture}
\end{minipage}
\caption{Illustration of the disk-disk problem for $\rho = 0.5$.}
\label{problem:disk-disk}
\end{figure}

The relevant weight function or pressure $f=M\chi_{_{B(0,\rho)}}$, where $M$ is a constant such that $|M|<4$.  The functional $I(\varphi)$ is mean coercive in the sense that there exists $\gamma>0$ such that 
\begin{align*}
    I(\varphi) & \geq \gamma \int_{\om}|\nabla \varphi|^2 \, \dx \quad \varphi \in W_0^{1,2}(\om,\R^2),
\end{align*}
which enables us to minimise $I(\cdot)$ in $W_{u_{0}}^{1,2}(B,\R^2)$ and so derive the Euler-Lagrange equation
\begin{align}\label{eltest}
    \int_{B}2\nabla u \cdot \nabla \varphi + M\chi_{_{B(0,\rho)}}\, \cof\nabla u\cdot \nabla\varphi \dx = 0 \quad \quad \varphi \in H^1_0(B,\R^2).
\end{align}

Nevertheless, it is still possible to show that the mean coercivity of $I(\cdot)$ is sufficient to improve the regularity of $W^{1,2}$ solutions of \eqref{eltest} to $C^{0,\alpha}$ for some $\alpha >0$, echoing the results of Morrey \cite[Theorem 4.3.1]{Mo66} and Giaquinta and Giusti \cite{GiaquintaGiusti1982}, for example, and enabling us to `join' pieces of the solution to \eqref{eltest} across the set $\partial B(0,\rho)$ where $f$ is discontinuous, leading in particular to \eqref{examplediskdisk}.  These and related results appeared in \cite{BKV25}.

The paper is organized as follows. In Section~\ref{insuproper}, we introduce and solve the so-called \emph{canonical insulation problem}, which consists of two rectangular regions where $f = c$ and $f = -c$, separated by a subdomain in which $f = 0$. The main result of this section is Theorem~\ref{Thm:insulation}. Proposition~\ref{first-order} further shows that solutions to the Euler--Lagrange equations are minimizers of the associated functional $E_M$ defined in \eqref{EM}. Moreover, uniqueness of the minimizer holds under the additional assumption that $E_M$ is mean coercive.

Section~\ref{thinsulate} addresses configurations in which the intermediate region $f = 0$ between $f = c$ and $f = -c$ is thinner. The main result of this section is Proposition~\ref{asymptotic}.

Finally, Section~\ref{numerics} presents numerical experiments which suggest that the inequalities in Theorem~\ref{Th:delta} and Proposition~\ref{asymptotic} could, in fact, be further strengthened.

\section{The canonical insulation problem}\label{insuproper}


We now again focus on the domain 
\begin{equation}\label{domain_insulation}
\Omega := R_{-2} \cup R_{-1} \cup R_1 \cup R_2,
\end{equation}
defined in \eqref{Omega} and shown in Figure~\ref{pic:pureinsulation}.

Let $c>0$ be constant, define the piecewise constant function $f$ by
\begin{align}\label{f_four_square} f(x):=-c\chi_{_{R_{-2}}} + c \chi_{_{R_{2}}}, \end{align}
and form the functional
\begin{align}\label{I_four_square} I\left(\varphi,\INS{-c}{0}{0}{c} \,\right)& := \int_{\om} |\nabla \varphi|^2+f(x) \det \nabla \varphi(x) \dx \\ & = 
\nonumber \int_{\om}|\nabla \varphi|^2 \dx 
- c \int_{R_{-2}} \det \nabla \varphi \dx
+ c\int_{R_2} \det \nabla \varphi \dx 
.
\end{align}


\begin{theorem}\label{Thm:insulation}
Assume that $|c|\le 4$. Then $I\left(\varphi,\INS{-c}{0}{0}{c} \,\right)\ge 0$ for every $\varphi\in W^{1,2}_0(\Omega;\R^2)$, and the upper bound of $4$ is sharp.
\end{theorem}

\begin{proof}
    Firstly, we may assume that $c\geq 0$, since if this is not the case then we set $\tilde{\varphi}(x_1,x_2)=\varphi(-x_1,x_2)$ and note that 
\begin{align*}
    I\left(\varphi,\INS{-c}{0}{0}{c} \,\right) = I\left(\tilde{\varphi},\INS{c}{0}{0}{-c} \,\right),
\end{align*}
    the point being that it does not matter whether it is $c$ or $-c$ that is attached to the region $R_{-2}$.  Next, we note that it is enough to show the statement for $c=4$.  Indeed, if we assume that $I\left(\varphi,\INS{-4}{0}{0}{4} \,\right) \geq 0$ for all $\varphi$ as above, then by writing
    \begin{align*}
       I\left(\varphi,\INS{-c}{0}{0}{c} \,\right) = \left\{\begin{array}{l l} \int_{\Omega}|\nabla \varphi|^2 \dx + c \Delta \quad &  \quad \mathrm{if}  \ \Delta \geq 0 \\[0.2cm]
       I\left(\varphi,\INS{-4}{0}{0}{4} \,\right) + (c-4)\Delta  \quad & \quad \mathrm{if} \ \Delta < 0,
       \end{array}\right.
    \end{align*}
where $$\Delta: = \int_{R_{2}}\det \nabla \varphi \, \dx - \int_{R_{-2}}\det \nabla \varphi \, \dx,$$
it is clear that $I\left(\varphi,\INS{-c}{0}{0}{c} \,\right) \geq 0$ for the same $\varphi$ and $0 \leq c \leq 4$.  Henceforth, we set $I(\varphi):=I\left(\varphi,\INS{-4}{0}{0}{4} \,\right)$ for brevity.
Moreover, it is easy to see that in order to minimise $I(\varphi)$ we can assume that 
    $\varphi$ is symmetric with respect to the line $x_1=0$, i.e., $\varphi(x_1,x_2)=\varphi(-x_1,x_2)$ for all $x_1\in [0,1]$.
Hence, it suffices to show that     
\begin{align}\label{half-domain}
   E(u)= \int_{R_{-2}} |\nabla u|^2 -4 \det\nabla u\,\dx +\int_{R_{-1}}|\nabla u |^2\,\dx\ge 0 
\end{align}    
for any $u\in W^{1,2}(\om;\R^2)$ such that $u=0$ on  $\partial(R_{-2}\cup R_{-1})\setminus\{x_1=0\}$, see Figure \ref{pic:pureinsulation_reduced}.
\begin{figure}[H]
\centering
\begin{minipage}{0.69\textwidth}
\centering
\begin{tikzpicture}[scale=0.8]
\node (A) at (-4,-2) {}; 
\node[right=2 of A.center] (B) {}; 
\node[right=4 of A.center] (O) {}; 

\node[above=4 of A.center] (A') {};
\node[above=4 of B.center] (B') {};
\node[above=4 of O.center] (O') {};

\filldraw[thick, top color=gray!30, bottom color=gray!30] (A.center) rectangle node{$-4$} (B'.center);

\node at ($(B)!0.5!(O')$) {0};

\node[below right=0.1 of A'.center] {$R_{-2}$};
\node[below right=0.1 of B'.center] {$R_{-1}$};

\draw[thick] (B.center) -- (O.center);
\draw[thick] (B'.center) -- (O'.center);

\draw[thick, dotted] (O.center) -- (O'.center);

\end{tikzpicture}
\end{minipage}
\caption{Part of the domain $\Omega$ divided into rectangles \( R_{-2} \) and \( R_{-1} \); boundary conditions of \(u\): solid lines = zero value, dotted line = free boundary.}
\label{pic:pureinsulation_reduced}
\end{figure}
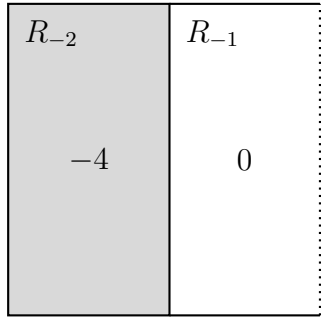

Condition \eqref{half-domain} is equivalent to 
$$
J(u,\psi):= \int_{R_{-2}} |\nabla u|^2 -4 \det\nabla u +|\nabla u+\nabla\psi|^2\,\dx \ge 0 $$

for all $u$ as above and all $\psi\in W^{1,2}(\om;\R^2)$ such that  $\psi=0$ on  $\partial(R_{-2}\cup R_{-1})\setminus\{x_1=-1\}$.

Given $u$ and $\psi $ we construct $\varphi$ as follows:

$$
\varphi(x):=
\begin{cases}
    u(x)  & \text{ if } x\in R_{-2},\\
 u(-x_1-1,x_2)+\psi(-x_1-1,x_2)    & \text{ if } x\in R_{-1},\\
 \varphi(-x_1,x_2) & \text{ if } x\in R_{1}\cup R_2,\end{cases}
$$
from which it follows that $E(\varphi)= J(u,\psi)$.  Keeping in mind that $|\cof A|=|A|$  and $2\det A=\cof A:A$ for any $A\in\R^{2\times 2}$,
we calculate
\begin{align*}
J(u,\psi)&= \int_{R_{-2}} |\cof\nabla u|^2 -4 \det\nabla u +|\nabla u|^2+|\nabla\psi|^2 + 2\nabla u:\nabla\psi\,\dx\\
&=\int_{R_{-2}} |\cof\nabla u-\nabla u|^2 +|\nabla\psi|^2 + 2\nabla u:\nabla\psi\,\dx\\
&= \int_{R_{-2}}|(\cof\nabla u-\nabla u)-\nabla\psi|^2+2(\cof\nabla u-\nabla u):\nabla \psi+2\nabla u:\nabla\psi\,\dx\\
&=\int_{R_{-2}}|\cof\nabla u-\nabla u-\nabla\psi|^2\,\dx\ge 0,
\end{align*}
where we have applied Lemma~\ref{Jon'slemma} below.  Finally, if $c>4$ (or, equivalently, $c<-4$), we can apply \cite[Proposition 3.4]{BKV23} in order to find $\varphi_0 \in W^{1,2}(\om,\R^2)$ such that $I(\varphi_0)<0$, so that the upper bound of $4$ in $|c|\leq 4$ is sharp.
\end{proof}

\begin{lemma}\label{Jon'slemma}
 It holds that $$\int_{R_{-2}}\cof\nabla u:\nabla\psi\,\dx=\int_{R_{-2}}\cof\nabla\psi :\nabla u\dx =0$$   for all $u$ and $\psi$ as in the proof of Theorem~\ref{Thm:insulation}.
\end{lemma}

\begin{proof}
    We get in view of the Piola identity ${\rm div}\,\,\cof\nabla u= {\rm div}\,\,\cof\nabla \psi=0$ that 
    $$
    \int_{R_{-2}}\cof\nabla u:\nabla\psi\,\dx=\int_{\partial R_{-2}} (\cof\nabla u)\nu\cdot\psi\,{\rm d}S=\int_{\partial R_{-2}} (\cof\nabla \psi)\nu\cdot u\,{\rm d}S , $$
    where $\nu$ is the unit outer normal to the boundary of $R_{-2}$.
    Note that $\psi=0$ on three faces of $R_{-2}$, and, on the face $x_1=-1$,  $u=0$, which implies that $(\cof\nabla u)\nu$  is zero along the boundary as well. 
    The statement follows. 
    \end{proof}

\begin{corollary}
If $J(u,\psi)=0$ if and only if $u=\psi=0$ in $R_{-2}$.
\end{corollary}
\begin{proof}
The "if" implication is trivial. We focus on the opposite one.
By Theorem 1, $J(u,\psi)=0$ only if 
\begin{align}\label{Martin'sidentity}\cof \, \nabla u - \nabla u = \nabla \psi \quad \mathrm{a.e. \ in} \ R_{-2}.
\end{align}
Taking the inner product of this expression with $\cof \, \nabla u$ and integrating over $R_{-2}$ gives 
\begin{align}\label{conformal}\int_{R_{-2}}|\nabla u|^2 - 2 \det \nabla u \, \dx = 0,
\end{align}
where we have again made use of Lemma \ref{Jon'slemma}. By Hadamard's pointwise inequality for matrices in $\R^{2 \times 2}$, \eqref{conformal} implies that for a.e.\@ $x$ in $R_{-2}$ $\cof \nabla u(x) = \nabla u(x)$, i.e. $\nabla u(x)$ is conformal.  Using this and \eqref{Martin'sidentity}, it follows that $\nabla \psi=0$ a.e.\@ in $R_{-2}$, and hence, by the boundary conditions, $\psi=0$.  Moreover, since $\nabla u$ is conformal then the function $f: z=x+iy \mapsto u_1(x,y)+iu_2(x,y)$ is holomorphic in $R_{-2}$.  Let $a$ be the midpoint of the line joining $-1-i/2$ and $-1/2-i/2$, and note that the function
\begin{align*}
    \widetilde{f}(z)& :=\left\{ \begin{array}{l l } f(z) & \mathrm{if} \ \im z \geq -1/2  \\
    \overline{f(\bar{z}-i)} & \mathrm{if} \ \im z \leq -1/2 
    \end{array}\right.
\end{align*}
is, by applying the boundary condition $f(z)=0$ for $z \in R_{-2}$ such that $\im z = -1/2$ together with the Schwarz reflection principle, holomorphic in a sufficiently small disk $D(a,r)$ about the point $a$.  The set $\{z\in D(a,r): \ \widetilde{f}(z)=0\}$ clearly contains an accumulation point, so by standard results it holds that $\widetilde{f}(z) = 0$ for $z$ in $D(a,r)$, and in particular that $f(z)=0$ for  $z$ in $D(a,r) \cap R_{-2}$.  It now follows that $f=0$ in $R_{-2}$, so $u=0$ there.

\end{proof}

\begin{remark}\label{BC}
    It follows from the proof of Lemma~\ref{Jon'slemma} that the statement of the lemma holds true if 
    on the boundary $u=0$ or $\psi=0$. Hence, $J(u,\psi)\ge 0$ in much more general situations. For instance,  we need only to assume that $\psi=0$ on the part of the boundary where $x_1=-1/2$.
\end{remark}
    
\subsection{Uniqueness of minimizers for polyconvex integrands}
    Let us put $R=R_{-2}\cup R_{-1}$. Consequently,  it follows for $E$ given in \eqref{half-domain} that  $E(u)\ge 0$ if $u\in W^{1,2}(R;\R^2)$  such that $u=0$  on $\partial R_2\setminus\{x:\, x_1=-1/2\}$.
Let us define for $u\in W^{1,2}(R;\R^2)$ and $M\in\R$
\begin{align}\label{EM}
    E_M(u)= \int_{R_{-2}} |\nabla u|^2 -M \det\nabla u\,\dx +\int_{R_{-1}}|\nabla u |^2\,\dx \end{align}

\begin{proposition}\label{first-order}
    Let $0\le M<4$ then 
    $$E_M(u)\ge \frac{4-M}{4}\int_{R} |\nabla u|^2\,\md x$$ for every $u\in W^{1,2}(R;\R^2)$  such that $u=0$  on $\partial R_2\setminus\{x:\, x_1=-1/2\}$, i.e., $E_M$ is mean coercive. 
    \end{proposition}

    \begin{proof}
        In view of Remark~\ref{BC} and the fact that $E_4=E$ given in \eqref{half-domain} we have that $E_4\ge 0$. 
        Finally, 
        $$E_4(u)=\frac4M E_M(u)-(4/M-1)\int_{R} |\nabla u|^2\,\md x\ge 0 $$
        and the result follows.
    \end{proof}

Consider $u_0\in W^{1,2}(R;\R^2)$ and define 
$$\mathcal{U}_{u_0} =\{u\in W^{1,2}(R;\R^2):\,  \textrm{ such that } u=u_0  \textrm{ on } \partial R_2\setminus\{x:\, x_1=-1/2\} \}. $$

If $\psi\in \mathcal{U}_{u_0}$ and $u\in \mathcal{U}_0$ then 
\begin{align}\label{minimizer}
E_M(\psi+ u)=E_M(\psi)+\langle E'(\psi),u\rangle +E_M(u)\ .
\end{align}

The next proposition shows that solutions of the Euler-Lagrange equations for $E_M$  with mixed boundary conditions are minimizers of $E_M$ despite the fact that the integrand is not convex.  This is an analogous result to \cite{neff}, see also \cite{john,spector,spector2} for uniqueness results in nonlinear elasticity. 

\begin{proposition}
    Let $0<M\le 4$ and let  $\psi\in\mathcal{U}_{u_0}$ be such that  $\langle E'(\psi),u\rangle=0$. Then $\psi$ is a minimizer 
    of $E_M$ on $\mathcal{U}_{u_0}$  which is unique for $M<4$.
\end{proposition}

\begin{proof}
It follows from \eqref{minimizer} and from mean coercivity of $E_M$ for $0\le M<4$.
\end{proof}

\subsection{Varying the width of the insulation layer}\label{thinsulate}

In this section we consider the effect of varying the width of the so-called insulation layer, which in Section \ref{insuproper} corresponded to the region $R_{-1} \cup R_{1}$ and which was of unit width. The intuition is that the thinner the insulation region, the smaller the range of $c$ for which the corresponding excess function is nonnnegative. 

Evidence for this is supplied by Theorem \ref{Th:delta} and Proposition \ref{asymptotic} below, a consequence of which is the result that if $N \in\N$ and the canonical domain is replaced with
$$\Omega_{N}:=\left(-\frac{1}{2}-\frac{1}{2N},\frac{1}{2}+\frac{1}{2N}\right) \times \left(-\frac{1}{2}, \frac{1}{2}\right),$$
consisting of the regions 
\begin{align*}R_{-2,N} & =\left(-\frac{1}{2}-\frac{1}{2N},-\frac{1}{2N}\right) \times \left(-\frac{1}{2},\frac{1}{2}\right) \ \mathrm{and} \\
R_{2,N} & =\left(\frac{1}{2N},\frac{1}{2}+\frac{1}{2N}\right) \times \left(-\frac{1}{2},\frac{1}{2}\right)\end{align*} which are separated by a `thin' insulation layer $\left(-\frac{1}{2N},\frac{1}{2N}\right) \times \left(-\frac{1}{2},\frac{1}{2}\right)$, and if 
$$f(x):= c \chi_{R_{2,N}} - c \chi_{R_{-2,N}},$$
then for all $\varphi \in W_0^{1,2}(\om_N,\R^2)$ we have
\begin{align}\label{inequalityN} \int_{\om_N} |\nabla u|^2 + f(x) \det \nabla u \, \dx \geq 0 \quad \mathrm{if}  |c| \lesssim 2 + \frac{2}{N}.\end{align} 
As the insulation region shrinks, we see that the upper bound in the sufficient condition $|c| \lesssim 2 + \frac{2}{N}$ approaches 2.  The extent to which this is necessary is investigated numerically in Section \ref{theorem5numerics} and, at least for reasonable values of $N$, there is evidence that the upper bound $2 + \frac{2}{N}$ is far from necessary other than in the case $N=1$, but with a similar (decreasing) trend.   

Using the symmetry of the domain $\om_N$, we can argue as we did for the canonical insulation problem that in order to prove \eqref{inequalityN} it is sufficient to show 
$$\int_{R_{-2,N}}|\nabla u|^2 - c \det \nabla u \, \dx + \int_{(-1/2N,0)\times (-1/2,1/2)} |\nabla u|^2 \geq 0$$
for all $u$ in $W^{1,2}(R_{-2,N}\cup (-1/2N,0)\times (-1/2,1/2);\R^2)$ such that $u=0$ on the boundary of that region other than along $\{0\} \times (-1/2,1/2)$.  Then, by a translation of $-\frac{1}{2}$ along the $x_1$ axis, we further argue that it is sufficient to prove 
$$\int_{R_{-2}}|\nabla u|^2 - c \det \nabla u \, \dx + \int_{R^\delta} |\nabla u|^2 \, \dx \geq 0$$
for all $u$ in $W^{1,2}(R_{-2} \cup R^\delta;\R^2)$ such that $u=0$ on $\partial(R_{-2}\cup R^{\delta})\setminus \{x:\, x_1=-1/2+\delta\}$. Here, $R^\delta=(-1/2,-1/2+\delta)\times(-1/2,1/2)$ and $R_{-2}$ is defined in Figure \ref{pic:pureinsulation}.

To that end, it follows from \cite{BKV23,mielke-sprenger} that $\forall\varepsilon>0$  there exists $u_\varepsilon\in W^{1,2}(R_{-2};\R^2)$, $u_\varepsilon=0$ on $\partial R_{-2}\setminus \{x_1=-1/2\}$ such that 
\begin{align}\label{ms}
    \int_{R_{-2}}|\nabla u_\varepsilon|^2 \,\md x < (2+\varepsilon)\int_{R_{-2}}\det\nabla u_\varepsilon\,\md x\ .
\end{align}
Let us denote the set of such maps $\mathcal{A}_\varepsilon$, i.e.,
$$\mathcal{A}_\varepsilon=\{u_\varepsilon\in W^{1,2}(R_{-2};\R^2):\, u_\varepsilon=0 \text{ on }\partial R_{-2}\setminus \{x_1=-1/2\}, \text{ and  \eqref{ms} holds }\}\ .$$
Hence, given $2<M<4$ we get for $\varepsilon>0 $ small enough
 \begin{align}\label{lower-upper bound}
   \frac{2-M}{2}\int_{R_{-2}}|\nabla u_\varepsilon|^2 \,\md x&\le \int_{R_{-2}}|\nabla u_\varepsilon|^2 \,\md x -M\int_{R_{-2}}\det\nabla u_\varepsilon\,\md x\nonumber\\\
&    <\int_{R_{-2}}|\nabla u_\varepsilon|^2 \,\md x -\frac{M}{2+\varepsilon} \int_{R_{-2}}|\nabla u_\varepsilon|^2 \,\md x\nonumber\\\
   &= \frac{2+\varepsilon-M}{2+\varepsilon} \int_{R_{-2}}|\nabla u_\varepsilon|^2 \,\md x\nonumber\\
   &<\frac{2+\varepsilon-M}{2} \int_{R_{-2}}|\nabla u_\varepsilon|^2 \,\md x.
\end{align} 

Let us extend $u_\varepsilon$ to the domain $$R^\delta=(-1/2,-1/2+\delta)\times(-1/2,1/2), \quad \text{where } \delta=(M-2)/4 $$ such that this extension has zero Dirichlet boundary conditions on the set $\{x_2=\pm 1/2\}$, see Figure  \ref{pic:pureinsulation_reduced_general}.

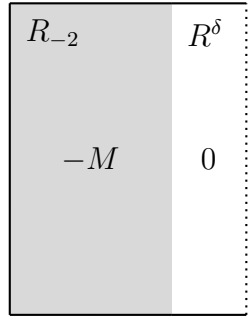
\begin{figure}[H]
\centering
\begin{minipage}{0.69\textwidth}
\centering
\begin{tikzpicture}[scale=0.8]
\node (A) at (-4,-2) {}; 
\node[right=2 of A.center] (B) {}; 
\node[right=3 of A.center] (O) {}; 

\node[above=4 of A.center] (A') {};
\node[above=4 of B.center] (B') {};
\node[above=4 of O.center] (O') {};

\fill[thick, top color=gray!30, bottom color=gray!30] (A.center) rectangle node{$-M$} (B'.center);
\draw[thick] (A.center) -- (A'.center);

\draw[thick] (A.center) -- (B.center);

\draw[thick] (A'.center) -- (B'.center);

\node at ($(B)!0.5!(O')$) {0};

\node[below right=0.1 of A'.center] {$R_{-2}$};
\node[below right=0.1 of B'.center] {$R^{\delta}$};

\draw[thick] (B.center) -- (O.center);
\draw[thick] (B'.center) -- (O'.center);

\draw[thick, dotted] (O.center) -- (O'.center);

\end{tikzpicture}
\end{minipage}
\caption{Part of the domain $\Omega$ divided into rectangles \( R_{-2} \) and \( R^{\delta} \); 
boundary conditions for \(u\): solid lines indicate zero value, the dotted line indicates a free boundary. 
Note that the horizontal lengths of \( R_{-2} \) and \( R^{\delta} \) are \(1/2\) and 
\(\delta = (M-2)/4\), respectively.}
\label{pic:pureinsulation_reduced_general}
\end{figure}

\begin{lemma}\label{lowerbound}
 If $2<M\le 3$, $2/(M-2)\in \mathbb{N}$ and $\delta=(M-2)/4$
 then 
\begin{align}\label{special}
  \int_{R^\delta}|\nabla u_\varepsilon|^2 \,\md x  \ge \frac{M-2-\varepsilon}{2} \int_{R_{-2}}|\nabla u_\varepsilon|^2 \,\md x\ .
\end{align} 
 
for every $u_\varepsilon\in\mathcal{A}_\varepsilon$ defined above and 
every $0<\varepsilon<M-2$.
    
\end{lemma}

\begin{proof}
    
Assume that 
\begin{align}
  \int_{R^\delta}|\nabla u_\varepsilon|^2 \,\md x  < \frac{M-2-\varepsilon}{2} \int_{R_{-2}}|\nabla u_\varepsilon|^2 \,\md x\ .
\end{align}
This implies that 
\begin{align}
   \int_{R_{-2}}|\nabla u_\varepsilon|^2 \,\md x -M\int_{R_{-2}}\det\nabla u_\varepsilon\,\md x+ \int_{R^\delta}|\nabla u_\varepsilon|^2 \,\md x<0\ .
\end{align}

Let us further assume that $2/(M-2)\in\mathbb{N}$.
 We extend $u_\varepsilon$ from $R^\delta$ to $R_{-1}$ by 2/(M-2)  flips along the vertical direction   so that 
 \begin{align}\label{upper-bound}
     \int_{R_{-1}}|\nabla u_\varepsilon|^2 \,\md x=\frac{2}{M-2} \int_{R^\delta}|\nabla u_\varepsilon|^2 \,\md x<
     \frac{(M-2-\varepsilon)}{(M-2)}\int_{R_{-2}}|\nabla u_\varepsilon|^2 \,\md x
 \end{align}

 We have 
 \begin{align}\label{M=4}
    \int_{R_{-2}}|\nabla u_\varepsilon|^2 -4\int_{R_{-2}}\det\nabla u_\varepsilon\,\md x <  \int_{R_{-2}}|\nabla u_\varepsilon|^2 -\frac{4}{2+\varepsilon} \int_{R_{-2}}|\nabla u_\varepsilon|^2\, \md x=\frac{\varepsilon-2}{\varepsilon+2}\int_{R_{-2}}|\nabla u_\varepsilon|^2\, \md x\ .
 \end{align}

 Putting together \eqref{M=4} and \eqref{upper-bound}
we get for $2<M\le 3$
\begin{align}
  \int_{R_{-2}}|\nabla u_\varepsilon|^2 -4\int_{R_{-2}}\det\nabla u_\varepsilon\,\md x+ \int_{R_{-1}}|\nabla u_\varepsilon|^2 \,\md x< \frac{\varepsilon\,(2M-6-\varepsilon)}{(\varepsilon+2)(M-2)}\int_{R_{-2}}|\nabla u_\varepsilon|^2\, \md x <0\ ,
\end{align}
 which contradicts Theorem~\ref{Thm:insulation} whenever $2<M\le 3$. 

 \end{proof}
Hence, adding \eqref{upper-bound} to both sides of \eqref{lower-upper bound} yields 
\begin{align*}
   \frac{-\varepsilon}{2}\int_{R_{-2}}|\nabla u_\varepsilon|^2 \,\md x&\le \int_{R_{-2}}|\nabla u_\varepsilon|^2 \,\md x -M\int_{R_{-2}}\det\nabla u_\varepsilon\,\md x+\int_{R^\delta}|\nabla u_\varepsilon|^2 \,\md x \\
\end{align*} 

and therefore 
\begin{align}\label{epsilon}
   0\le  \int_{R_{-2}}\frac{2+\varepsilon}{2}|\nabla u_\varepsilon|^2 \,\md x -M\int_{R_{-2}}\det\nabla u_\varepsilon\,\md x+\int_{R^\delta}|\nabla u_\varepsilon|^2 \,\md x 
\end{align} 
whenever $u_\varepsilon\in\mathcal{A}_\varepsilon.$

\begin{theorem} \label{Th:delta}
    Let $2\le M\le 3$ such that $2/(M-2)\in\mathbb{N}$ for $M>2$. Let further $\delta=(M-2)/4$  Then 
    \begin{align}\label{thm2-ineq}
   \frac{\int_{R_{-2}}|\nabla u|^2 \,\md x}{2\int_{R_{-2}}\det\nabla u\,\md x} \int_{R_{-2}}|\nabla u|^2 \,\md x-M\int_{R_{-2}}\det\nabla u\,\md x+ \int_{R^\delta}|\nabla u|^2 \,\md x\ge 0\ 
\end{align}
for every $u\in W^{1,2}(R_{-2}\cup R^\delta;\R^2)$, $u=0$ on $\partial ( R_{-2}\cup R^\delta)\setminus \{x:\, x_1=-1/2+\delta\}$ such that $\int_{R_{-2}}\det\nabla u\,\md x>0$. 
\end{theorem}

\begin{proof} If $M=2$ then $\delta=0$ and the statement follows from the pointwise Hadamard inequality.
    Assume that $M>2$ and that the statement is not true, i.e., that there is $u$ admissible such that  \eqref{thm2-ineq} does not hold. 
    Then inevitably 
\begin{align}
   \int_{R_{-2}}|\nabla u|^2 \,\md x -M\int_{R_{-2}}\det\nabla u\,\md x<0\ .
   \end{align}
   This means that there is $2+\varepsilon<M$ such that 
\begin{align}
   \frac{M}{2+\varepsilon}\int_{R_{-2}}|\nabla u|^2 \,\md x -M\int_{R_{-2}}\det\nabla u\,\md x<0\ .
   \end{align}   
 However, then $u\in\mathcal{A}_\varepsilon$ where $2+\varepsilon > \int_{R_{-2}}|\nabla u|^2 \,\md x /\int_{R_{-2}}\det\nabla u\,\md x$. Consequently, \eqref{epsilon} holds for this $u$ and $\varepsilon$ instead of $u_\varepsilon$. 
 Taking the infimum over $\varepsilon$ gives the result.
    
\end{proof}

\begin{remark}\label{expansion_delta}
We believe that the factor $\int_{R_{-2}}|\nabla u|^2 \,\md x /(2\int_{R_{-2}}\det\nabla u\,\md x)\ge 1$ in Theorem~\ref{Th:delta} can be omitted, and the numerical tests below confirm it.  The admissible pairs of Theorem~\ref{Th:delta} are exactly
\[
(\delta, M)=\Bigl(\tfrac{1}{2k},\,2+\tfrac{2}{k}\Bigr), \qquad k\in\mathbb{N}.
\]
For $k=1$, one obtains the setup of Theorem~\ref{Thm:insulation}. 
\end{remark}
In the spirit of Theorem \ref{Th:delta}, we have the following result.  
\begin{proposition}\label{asymptotic}
Let $k \in \N$ and set
\begin{align}\label{mk}M_k:=1-\frac{1}{k}+\left(\frac{1}{k^2}+\frac{6}{k}+1\right)^{\frac{1}{2}}\end{align}
and set $\delta:=\frac{1}{2k}$.  Then for every $u\in W^{1,2}(R_{-2}\cup R^\delta;\R^2)$ such that $u=0$ on $\partial ( R_{-2}\cup R^\delta)\setminus \{x_1=-1/2+\delta\}$, it holds that 
\begin{align}\label{marmite}
\int_{R_{-2}}|\nabla u|^2 - M_k \det \nabla u \dx + \int_{R^{\delta}}|\nabla u|^2 \,\dx \geq 0.
\end{align}
\end{proposition}
\begin{proof}
    Assume for a contradiction that \eqref{marmite} is false for one of the $u$ in the class described.  Then $u$ is not zero and
    \begin{align}
   \label{firstline}    \int_{R^{\delta}} |\nabla u|^2 \, \dx & < \int_{R_{-2}} M_k \det \nabla u - |\nabla u|^2 \, \dx \\ \nonumber
       & \leq \int_{R_{-2}} M_k \det \nabla u - 2 |\det \nabla u| \, \dx \\ \nonumber
       & \leq  (M_k-2)\int_{R_{-2}} |\det \nabla u| \, \dx \\
\label{marmalade}     \int_{R^{\delta}} |\nabla u|^2 \, \dx  & \leq   \frac{M_k-2}{2}\int_{R_{-2}} |\nabla u|^2 \, \dx,
    \end{align}
where we have used Hadamard's pointwise inequality $|\nabla u|^2 \geq 2 |\det \nabla u|$ twice.  Now extend $u$ to a function $\tilde{u}$ in $W^{1,2}(R_{-2} \cup R_{-1})$ by a sequence of $k-1$ reflections of $u\arrowvert_{_{R^{\delta}}}$ in vertical lines that cross the $x_1-$axis at points $-1/2+n \delta$ for $n=1,\ldots,k-1$, and note that by \eqref{marmalade} and because $\tilde{u}=u$ on $R_{-2}$, we have
\begin{align}\label{jam}\int_{R_{-1}}|\nabla \tilde{u}|^2 \, \dx < \frac{k(M_k-2)}{2}\int_{R_{-2}} |\nabla \tilde{u}|^2 \, \dx. 
\end{align}

Note also that by rearranging \eqref{firstline} it follows that 
\begin{align}
    \label{lemoncurd} \int_{R_{-2}}\det \nabla \tilde{u} \, \dx > \frac{1}{M_k} \int_{R_{-2}} |\nabla \tilde{u}|^2 \, \dx. 
\end{align}

Now consider the quantity
$$I(\tilde{u})=\int_{R_2}|\nabla \tilde{u}|^2 - 4\det \nabla \tilde{u} \, \dx + \int_{R_{-1}} |\nabla \tilde{u}|^2 \, \dx,$$ 
and note that by \eqref{jam} and \eqref{lemoncurd}
\begin{align*}
  I(\tilde{u}) & \leq \int_{R_{-2}} |\nabla \tilde{u}|^2 -\frac{4}{M_k}|\nabla \tilde{u}|^2 + \frac{k(M_k-2)}{2} |\nabla \tilde{u}|^2 \, \dx \\
  & = \left( 1 -\frac{4}{M_k} + \frac{k(M_k-2)}{2} \right)\int_{R_{-2}}|\nabla \tilde{u}|^2 \, \dx. 
\end{align*}
But $M_k$ given by \eqref{mk} is such that 
$$  1 -\frac{4}{M_k} + \frac{k(M_k-2)}{2} = 0,$$
which implies $I(\tilde{u})=0$.  By Theorem~\ref{Thm:insulation}, this is possible only if $\tilde{u}$ is zero, which is contrary to our initial assumption. 
\end{proof}

We remark that for large $k$ the constants $M_k$ given by \eqref{mk} obey
$$M_k \sim 2 + \frac{2}{k} - \frac{1}{2k^2} + o\left(\frac{1}{k^2}\right),$$
so that we recover, albeit asymptotically, the result of Theorem \ref{Th:delta}, and without the need for additional hypotheses. 

\section{Numerical experiments}\label{numerics}

The functional \eqref{I_four_square} reads
\begin{equation}\label{I_four_square_fem}
I(\varphi)
=
\int_{\Omega}
\nabla \varphi : \nabla \varphi
+
\frac12\, f \, \nabla \varphi : \operatorname{cof}\nabla \varphi
\,\mathrm{d}x,
\end{equation}
where \( : \) denotes the Frobenius inner product on \( \mathbb{R}^{2\times2} \).
Using the identity \( A:\operatorname{cof}A = 2\det A \), the second term can be interpreted as a determinant contribution.

Let \( V_h \subset W^{1,2}(\Omega;\mathbb{R}^2) \) be a finite element space of continuous, piecewise affine vector fields over a triangular mesh aligned with the jump set of \(f\); see Figure~\ref{mesh}.

\begin{figure}[htbp]
    \centering
    \includegraphics[width=0.99\textwidth]{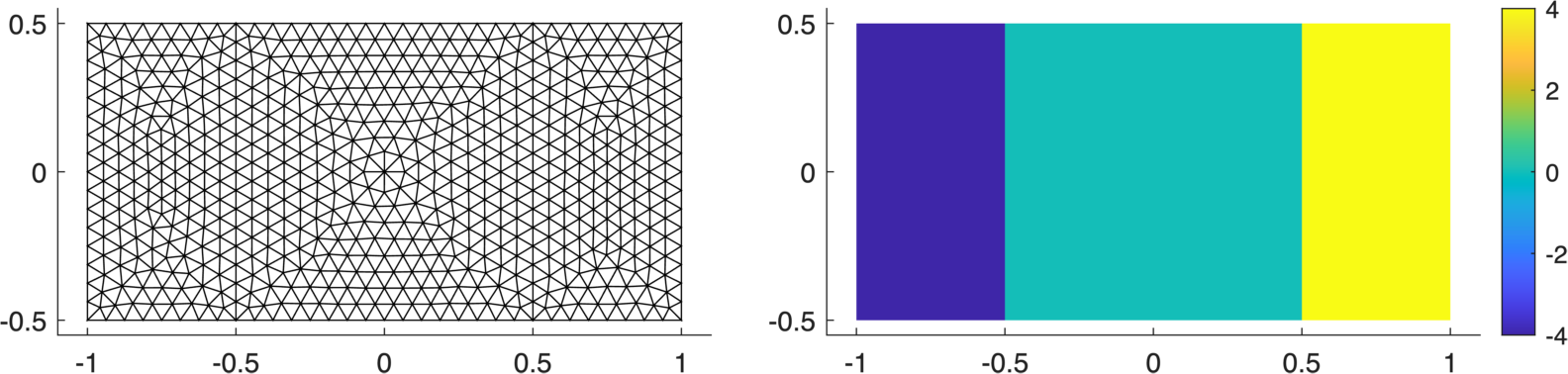}
    \caption{Example of a level 3 triangular mesh of \( \Omega \) (left) and the distribution of \( f \) (right).}
    \label{mesh}
\end{figure}

\noindent
The approximation \( \varphi_h \in V_h \) is represented as
\[
\varphi_h = \sum_{i=1}^N v_i \phi_i,
\qquad
\mathbf v=(v_1,\dots,v_N)^\top\in\mathbb{R}^N,
\]
where \( \{\phi_i\}_{i=1}^N \) is the standard finite element basis of globally continuous, piecewise affine vector-valued functions associated with the mesh nodes. These basis functions satisfy the nodal interpolation property \( \phi_i(x_j)=\delta_{ij} \), so that the coefficient vector \( \mathbf v \) represents the nodal values of the discrete deformation:
\[
\varphi_h(x_i) = v_i.
\]

Since the basis functions are affine on each triangle, their gradients are piecewise constant, and thus
\[
\nabla \varphi_h = \sum_{i=1}^N v_i \nabla \phi_i.
\]

Substituting this expansion into \eqref{I_four_square_fem} and exploiting bilinearity, we define the stiffness matrices (cf.~\cite{BKV25})
\[
(K_1)_{ij}
=
\int_{\Omega} \nabla \phi_j : \nabla \phi_i\,\mathrm{d}x,
\qquad
(K_2)_{ij}
=
\int_{\Omega}
f\, \nabla \phi_j : \operatorname{cof}\nabla \phi_i\,\mathrm{d}x,
\]
which leads to the quadratic representation
\begin{equation}\label{quadratic_form}
I(\varphi_h)
=
\mathbf v^{\top}
\bigl(K_1+\tfrac12 K_2\bigr)\,
\mathbf v
=:
\mathbf v^{\top}
A \,
\mathbf v.
\end{equation}

\noindent
The matrix \( A \) is symmetric. The positivity of its minimal eigenvalue \( \lambda_{\min}(A) \) guarantees that the quadratic form \eqref{quadratic_form} is positive definite and hence confirms Theorem~\ref{Thm:insulation}, while \( \lambda_{\min}(A) < 0 \) indicates a violation.

Equivalently, positive definiteness of \( A \) can be verified by the existence of a Cholesky factorization \( A = LL^\top \). In the computations, failure of the Cholesky decomposition provides a practical indicator of the loss of positive definiteness.

\subsubsection{Numerical verification  of Theorem \ref{Thm:insulation}}
The matrix \( A \) is assembled for a sequence of uniformly refined triangulations of the domain \( \Omega \) defined in~\eqref{domain_insulation}. The right-hand side function \( f \) is given by~\eqref{f_four_square}, with parameters \( c = 4 \) or \( c = 4.1 \).

For each refinement level, the minimal eigenvalue \( \lambda_{\min}(A) \) is computed. The resulting values are summarized in Table~\ref{tab:min_eigenvalue_levels}.

\begin{table}[H]
\centering
\begin{tabular}{
    S[table-format=1.0]
    @{\hspace{18pt}}
    S[table-format=5.0]
    @{\hspace{18pt}}
    S[table-format=5.0]
    @{\hspace{22pt}}
    S[scientific-notation = true, table-format=1.3e1]
    @{\hspace{24pt}}
    S[scientific-notation = true, table-format=1.3e1]
}
\hline
{Level} &
{Triangles} &
{Nodes} &
{\(\lambda_{\min}\) ($c=4$)} &
{\(\lambda_{\min}\) ($c=4.1$)} \\
\hline
1 & 76     & 51     & 6.024e0   &  1.050e0  \\
2 & 294    & 172    & 2.019e-1  &  1.885e-1 \\
3 & 1198   & 648    & 3.095e-2  &  2.267e-2 \\
4 & 4678   & 2436   & 4.360e-3  & -2.722e-4 \\
5 & 18870  & 9628   & 5.310e-4  & -2.650e-3 \\
6 & 75532  & 38151  & 5.864e-5  & -3.037e-3 \\
\hline
\end{tabular}
\caption{Minimal eigenvalues \( \lambda_{\min}(A) \) across refinement levels.}
\label{tab:min_eigenvalue_levels}
\end{table}

From the results, we observe that for \( c = 4 \), all eigenvalues \( \lambda_{\min}(A) \) remain positive and monotonically decrease towards zero as the mesh is refined. This behavior is in full agreement with Theorem~\ref{Thm:insulation}.

In contrast, for \( c = 4.1 \), the minimal eigenvalues become negative starting from refinement level 4, and their magnitude increases with further refinement. This clearly violates the conditions of Theorem~\ref{Thm:insulation}. This behavior indicates a loss of coercivity of the discrete operator once \( c \) exceeds a critical threshold.

To illustrate the structure of the corresponding eigenmode, the eigenvector \( \varphi_{\min} \) associated with \( \lambda_{\min}(A) \) is shown in Figure~\ref{fig:min_eigenvector}.

\begin{figure}[H]
    \centering
    \vspace{0.3cm}
    \includegraphics[width=\textwidth]{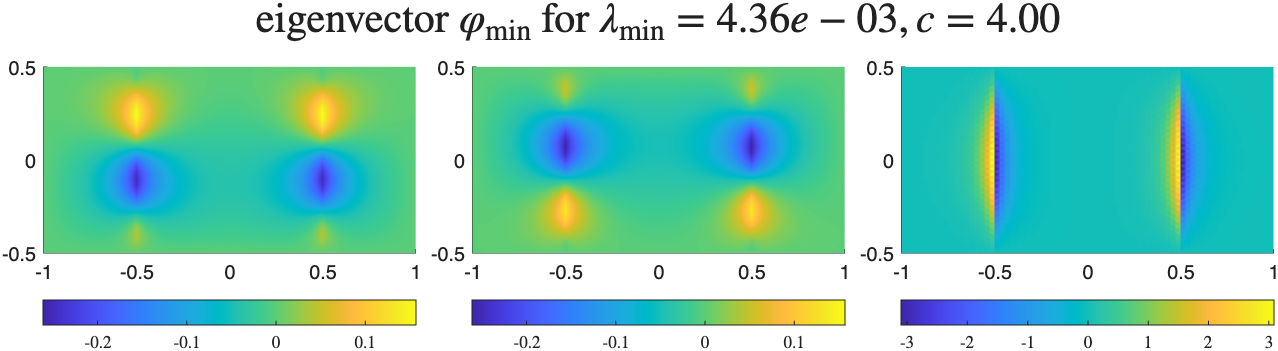}
    \caption{Components of \( \varphi_{\min} \) (first two columns) and \( \det(\nabla \varphi_{\min}) \) (third column) for \( c = 4 \) on the level 4 mesh with 4678 triangles.}
    \label{fig:min_eigenvector}
\end{figure}

\subsubsection{Numerical verification of Theorem \ref{Th:delta}} \label{theorem5numerics}

Admissible pairs \( (\delta, M) \) from Theorem~\ref{Th:delta} are listed in Remark~\ref{expansion_delta}. 

In the numerical evaluation, the matrix \( A \) is repeatedly assembled within a bisection procedure for varying values of \( M \). 
\begin{figure}
    \centering
    \includegraphics[width=0.8\textwidth]{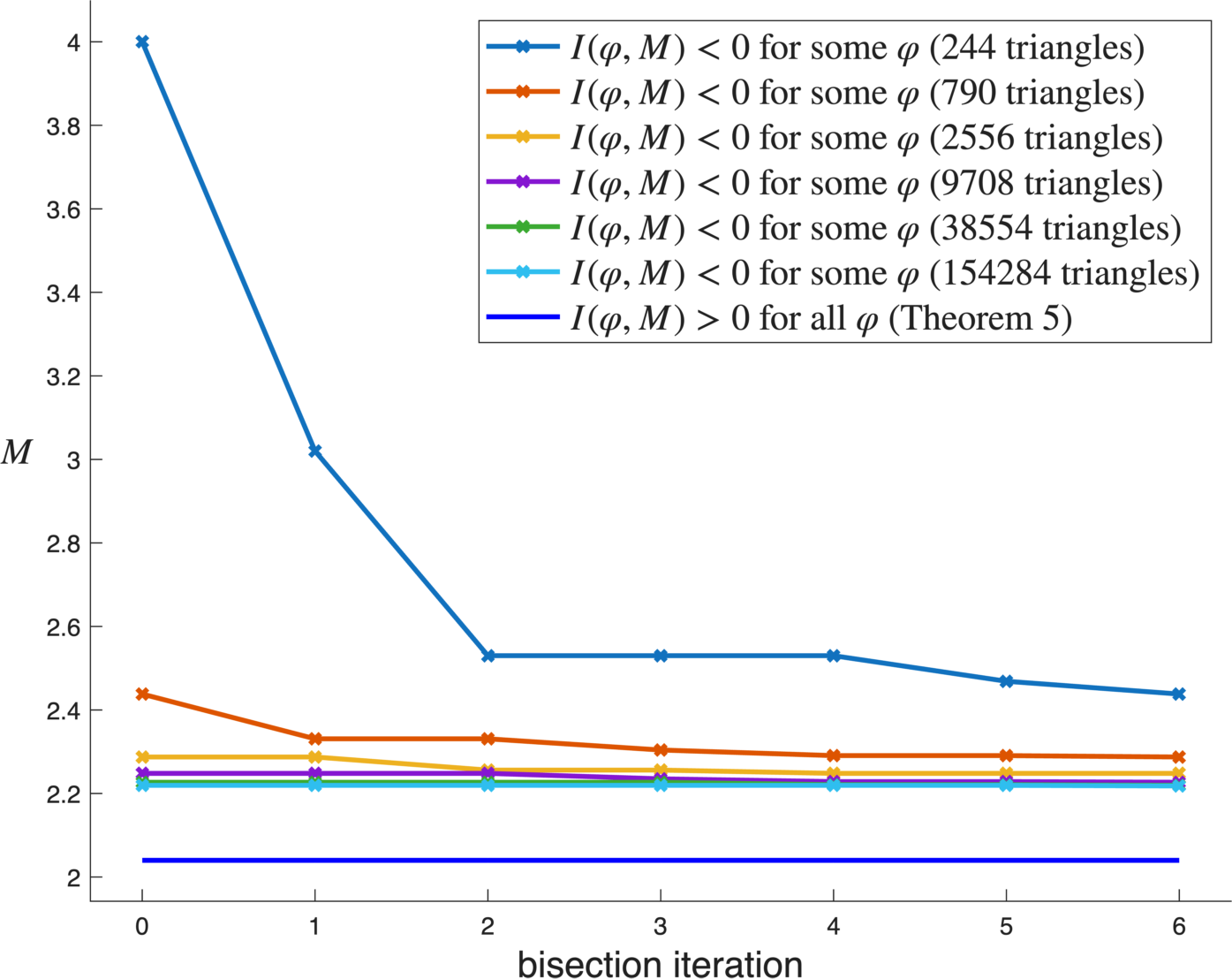}
    \vspace{-0.2cm}
    \caption{Bisection procedure for \( k = 50 \).}
    \label{fig:min_eigenvector_negative}
\end{figure}
The goal is to determine the minimal value \( M_{\text{num}} > M \) such that the corresponding minimal eigenvalue \( \lambda_{\min}(A) \) becomes negative; see Figure~\ref{fig:min_eigenvector_negative}. This value provides a numerical upper bound for the parameter \( M \) in Theorem~\ref{Th:delta}.

The results are summarized in Table~\ref{table:delta_M}. We observe a good agreement between the theoretical and numerical values for larger values of \( k \), where the inner domain \( R^{\delta} \) becomes very thin. For smaller values of \( k \), a noticeable gap appears, indicating potential room for improvement in Theorem~\ref{Th:delta}.

\begin{table}
\centering
\begin{tabular}{
    r
    @{\hspace{1.2cm}}
    c
    @{\hspace{1.2cm}}
    r
    @{\hspace{1.2cm}}
    r
}
\hline
$k$ & $\delta$ & $M$ (theoretical) & $M_{\text{num}}$ (upper bound) \\
\hline
1   & 0.5   & 4.000 & $< 4.024$ \\
2   & 0.25  & 3.000 & $< 3.953$ \\
3   & 0.167 & 2.667 & $< 3.762$ \\
4   & 0.125 & 2.500 & $< 3.547$ \\
5   & 0.100 & 2.400 & $< 3.371$ \\
10  & 0.050 & 2.200 & $< 2.842$ \\
20  & 0.025 & 2.100 & $< 2.480$ \\
30  & 0.017 & 2.067 & $< 2.341$ \\
40  & 0.013 & 2.050 & $< 2.267$ \\
50  & 0.010 & 2.040 & $< 2.220$ \\
100 & 0.005 & 2.020 & $< 2.121$ \\ 
200 & 0.003 & 2.010 & $< 2.067$ \\ 
300 & 0.002 & 2.007 & $< 2.047$ \\ 
$\infty$ & 0 & 2.000 & $< 2.031$ \\ 
\hline
\end{tabular}
\caption{Comparison between theoretical values of \( M \) from Theorem~\ref{Th:delta} and numerically obtained upper bounds \( M_{\text{num}} \).}
\label{table:delta_M}
\end{table}

\subsubsection{Homogeneous Dirichlet - Neumann boundary condition and oscilations}
The numerical approach also enables the investigation of mixed boundary conditions. 
\begin{figure}
\centering
\begin{minipage}{0.66\textwidth}
\centering
\begin{tikzpicture}[scale=0.8]
\node (A) at (-4,-2) {};
\node[right=4 of A.center] (B) {};
\node[right=6 of A.center] (O) {};

\node[above=4 of A.center] (A') {};
\node[above=4 of B.center] (B') {};
\node[above=4 of O.center] (O') {};

\fill[top color=gray!30, bottom color=gray!30] (A.center) rectangle node{$-4$} (B'.center);

\draw[thick] (A.center) -- (A'.center);
\draw[thick] (A.center) -- (B.center);
\draw[thick] (A'.center) -- (B'.center);

\node at ($(B)!0.5!(O')$) {0};

\node[below right=0.1 of A'.center] {$R_{-}$};
\node[below right=0.1 of B'.center] {$R^{\delta}$};

\draw[thick, dotted] (B.center) -- (O.center);
\draw[thick, dotted] (B'.center) -- (O'.center);
\draw[thick, dotted] (O.center) -- (O'.center);
\end{tikzpicture}
\end{minipage}
\hfill
\begin{minipage}{0.30\textwidth}
\small
\begin{align*}
R_{-} &= (-1,0)\times\left(-\tfrac12,\tfrac12\right), \\
R^{\delta} &= (0,\delta)\times\left(-\tfrac12,\tfrac12\right).
\end{align*}
\end{minipage}
\caption{Decomposition of the domain \( \Omega \) into subregions \( R_{-} \) and \( R^{\delta} \).
Solid lines indicate Dirichlet boundary conditions, while dotted lines correspond to Neumann (free) boundaries.}
\label{pic:pureinsulation_neumann}
\end{figure}
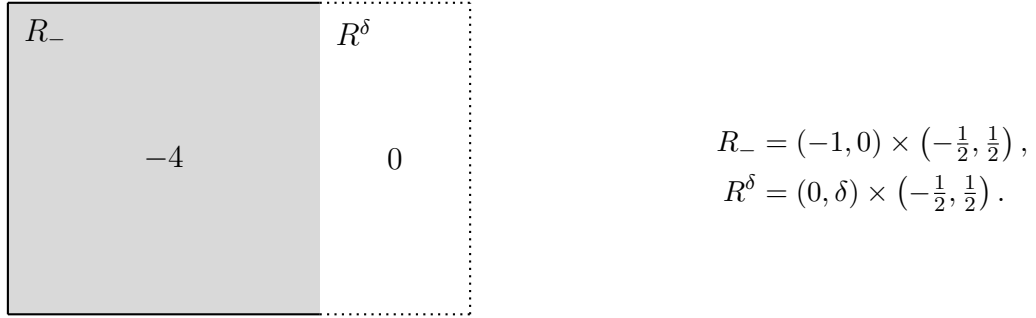
As a representative example, we consider the rectangular domain 
\[
\Omega = (-1,\delta)\times\left(-\tfrac12,\tfrac12\right), \qquad \delta \geq 0.
\] 
shown in Figure \ref{pic:pureinsulation_neumann}.
Homogeneous Dirichlet boundary conditions are imposed on
\[
\Gamma_D := \{ x \in \partial\Omega : x_1 \leq 0 \},
\]
and homogeneous Neumann boundary conditions on the remaining part
\[
\Gamma_N := \partial\Omega \setminus \Gamma_D.
\]
As \( \delta \to 0^+ \), the Neumann boundary degenerates to the vertical edge \( x_1 = 0 \).

The minimal eigenvalue and the associated eigenvector for \( \delta = 1 \) are shown in Figure~\ref{fig:min_eigenvector_neumann_full}. We observe a pronounced concentration of the determinant near the interface between \( R_{-} \) and \( R^{\delta} \).

To further investigate this effect, the domain is partitioned into layers in the \( x_1 \)-direction, and the contribution of the gradient term to the total energy is evaluated and normalized to one. Figure~\ref{fig:gradient_distribution} shows that the gradient part of the energy is likewise concentrated in a narrow region near the same interface.

\begin{figure}
    \centering
    \includegraphics[width=\textwidth]{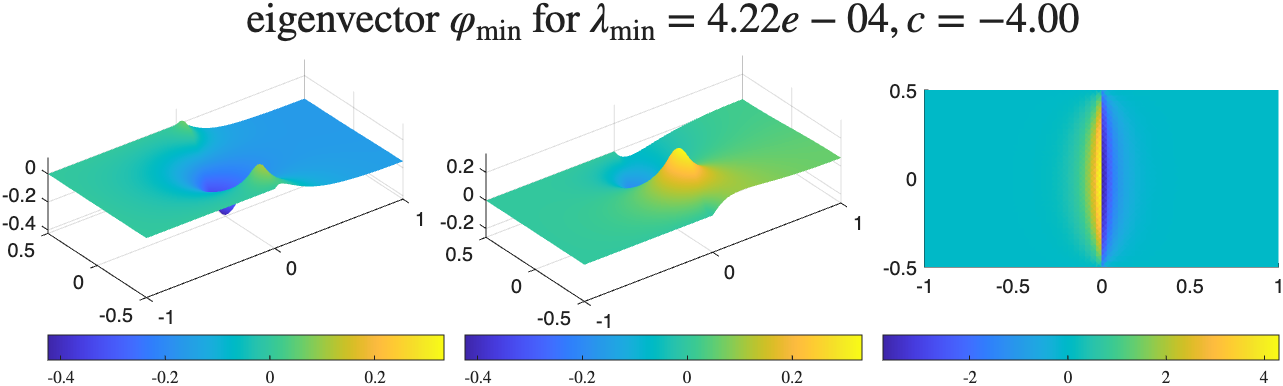}
    \caption{Components of \( \varphi_{\min} \) (first two columns) and \( \det(\nabla \varphi_{\min}) \) (third column) for \( c = -4 \) and \( \delta = 1\) on the mesh with 8192 triangles.}
    \label{fig:min_eigenvector_neumann_full}
    \centering
    \includegraphics[width=0.69\textwidth]{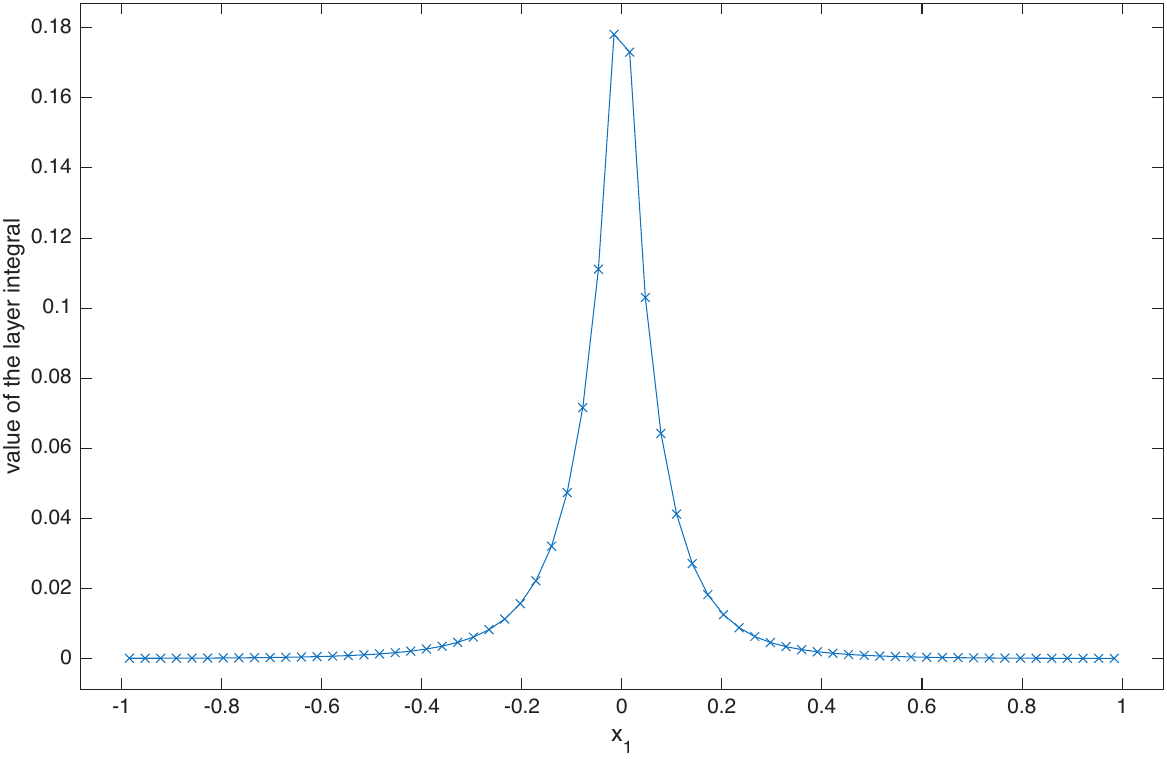}
    \caption{Distribution of the gradient contribution $|\nabla \varphi_{\min}|^2$ across \( x_1 \)-layers for \( c = -4 \) and \( \delta = 1 \) at level~4 mesh with 8192 triangles. Altogether, there are 64 \( x_1 \)-layers.}
    \label{fig:gradient_distribution}
    \centering
    \includegraphics[width=0.8\textwidth]{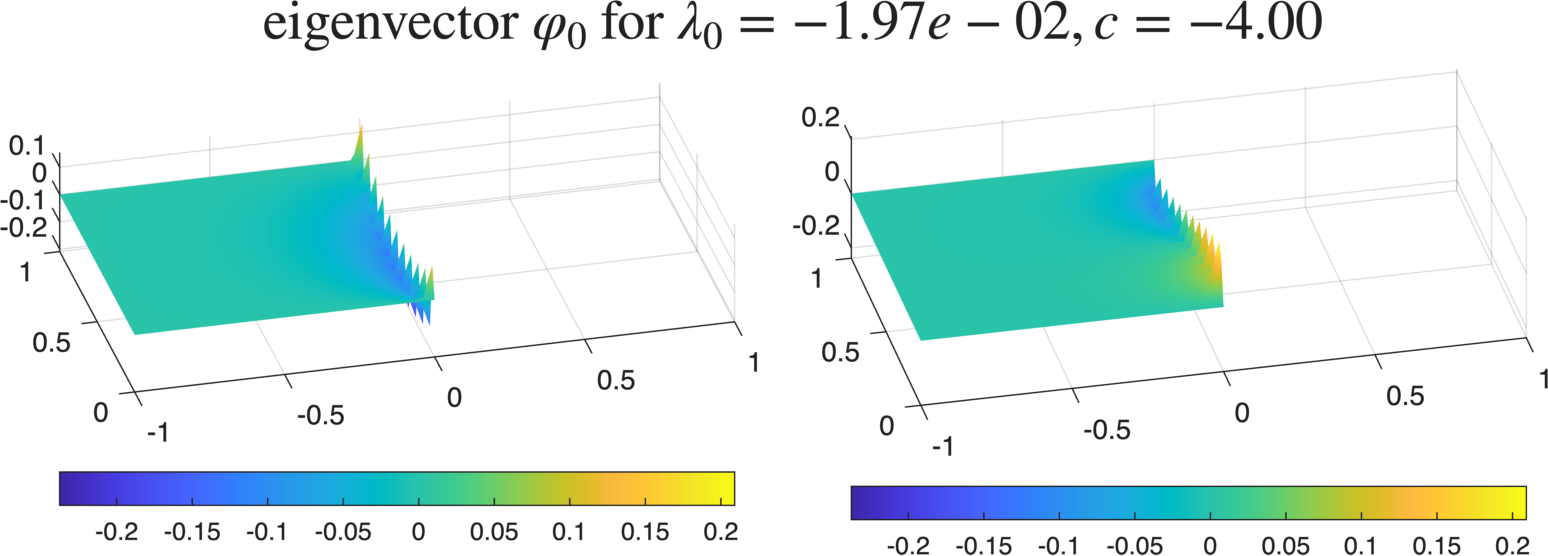}
    \caption{Components of \( \varphi_{\min} \) for \( c = -4 \) and \( \delta = 0\) on the level 4 mesh with 4032 triangles.
    }
    \label{fig:min_eigenvector_neumann}
\end{figure}
In the limiting case \( \delta = 0 \), the minimal eigenvector localizes near the interface \( x_1 = 0 \), exhibiting boundary-dominated behavior as a consequence of the degeneration of the Neumann boundary; see Figure~\ref{fig:min_eigenvector_neumann}.

\vspace{0.5cm}
\noindent
{\bf Data availability statement}\\[1mm]
\noindent The MATLAB code used to generate the numerical results is available for download at
\begin{center}
\url{https://www.mathworks.com/matlabcentral/fileexchange/130564} .
\end{center}
\bigskip 
\noindent
{\bf Acknowledgment}\\[1mm]
    \noindent MK and JV were partially supported by the GA\v{C}R project 23-04766S; MK was also supported by the subsequent GA\v{C}R project 26-21585K. The authors thank JB and the Department of Mathematics at the University of Surrey for their hospitality. JB gratefully acknowledges the hospitality of MK, JV, and the \'{U}TIA, Czech Academy of Sciences, during his visits.

\nocite{*}


\begin{thebibliography}{99}


\bibitem{Ba77} J. M. Ball: Convexity conditions and existence theorems in nonlinear elasticity.  {\it Arch. Rat. Mech. Anal.},  {\bf 63}, no. 4 (1977), 337--403.
J. M. Ball, J. Marsden. Quasiconvexity at the boundary, positivity of the second variation and elastic stability. {\it Arch.~Rat.~Mech.~Anal.} {\bf 86} (1984), 251--277. 


\bibitem{Bevan-14}
J. Bevan.  On double-covering stationary points of a constrained Dirichlet energy. {\it Ann. Inst. H. Poincar$\acute{\textrm{e}}$ Anal. Non Lin$\acute{\textrm{e}}$aire} {\bf 31} (2014),  391--411.


\bibitem{BKV23} 
J. Bevan, M. Kru\v{z}\'{i}k and J. Valdman. Hadamard's inequality in the mean.  Nonlinear Analysis {\bf(243)}, 
2024. https://doi.org/10.1016/j.na.2024.113523

\bibitem{BKV25} 
J. Bevan, M. Kru\v{z}\'{i}k and J. Valdman. New applications of Hadamard-in-the-mean inequalities to incompressible variational problems, Calculus of Variations and Partial Differential Equations {\bf(64)}, 259 (2025).

\bibitem{dacorogna} B. Dacorogna. {\it Direct Methods in the Calculus of Variations.} 2nd ed. Springer Science+Business Media, 2008. 

B. Dacorogna, P. Marcellini. {\it Implicit Partial Differential Equations.}  Progress in nonlinear
differential equations and their applications 37, Birkh\"{a}user, Boston, 1999.


\bibitem{BD22} M. Dengler, J. Bevan. A uniqueness criterion and a counterexample to regularity in an incompressible variational problem. Arxiv:2205.07694, 2022. 





I. Fonseca, S. M\"{u}ller, P. Pedregal. Analysis of concentration and oscillation effects generated by gradients. {\it SIAM J. Math. Anal.} {\bf 29} (1998), 736--756.

\referee

\bibitem{GiaquintaGiusti1982} M. Giaquinta and E. Giusti.  On the regularity of minima of variational integrals. {\it Acta Math.} {\bf 148} (1982), 31--46.

T. Iwaniec. Nonlinear Cauchy-Riemann operators in $\R^n$
{\it Trans. AMS} {\bf 354} (2002),   1961-–1995.

\bibitem{iwaniec-lutoborski-ARMA}
 T. Iwaniec, A. Lutoborski. Integral estimates for null Lagrangians. {\it Arch. Ration. Mech. Anal.} {\bf  125} (1993), 25--79.

 \bibitem{iwaniec-lutoborski-SIMA}
 T. Iwaniec, A. Lutoborski. Polyconvex functionals for nearly conformal deformations. {\it SIAM J. Math. Anal.}
 {\bf 27} (1996), 609--619.

 
 \bibitem{john}
 F. John. Uniqueness of non-linear elastic equilibrium for prescribed boundary displacements and sufficiently small strains. {\it Comm. Pure Appl. Math.} {\bf XXV} (1972), 617--634. 




\bibitem{mielke-sprenger}
A. Mielke, P. Sprenger. Quasiconvexity at the boundary and a simple variational formulation of Agmon's condition. {\it J. Elast.} {\bf 51} (1998),  23--41.


\bibitem{Mo66}  C. B. Morrey, Jr.  Multiple integrals in the calculus of variations. Reprint of the 1966 edition. Classics in Mathematics. Springer-Verlag, Berlin, 2008. 

\bibitem{neff}
D.Y. Gao, P. Neff, I. Roventa, C. Thiel. On the convexity of nonlinear elastic energies in  the right Cauchy-Green tensor. {\it J. Elasticity} {\bf 127} (2017), 303--308.


\bibitem{spector}
J. Sivaloganathan, S.J. Spector. On the uniqueness of energy einimizers in finite elasticity. {\it J. Elast.} {\bf 133} (2018), 73--103. 

\bibitem{spector2}
D.E. Spector, S.J. Spector. Uniqueness of equilibrium with sufficiently small strains in finite elasticity. {\it Arch. Ration. Mech. Anal.} {\bf 233} (2019), 409--449.\ 
\end{thebibliography}

\end{document}